\documentclass[a4paper, 10pt]{amsart}

\usepackage[latin1]{inputenc}
\usepackage[swedish, english]{babel}
\usepackage{epsfig}
\usepackage{float}
\usepackage{subfigure}
\usepackage{moreverb}

\usepackage{afterpage}		
\usepackage{latexsym}
\usepackage{amsmath}
\usepackage{amssymb}

\usepackage[numbers,sort&compress]{natbib}

\renewcommand{\appendix}{
\setcounter{section}{0}
\renewcommand{\thesection}{\Roman{section}}
\vspace{0.5cm}

{\Large{\bf APPENDIX}}}

\newtheorem{theorem}{Theorem}[section]

\newtheorem{proposition}[theorem]{Proposition}

\newtheorem{definition}[theorem]{Definition}

\newenvironment{remark}[1][Remark]{\begin{trivlist}
\item[\hskip \labelsep {\bfseries #1}]}{\end{trivlist}}

\newcommand{\Pp}{\mathbb{P}}

\newcommand{\OO}{\mathcal{O}}

\newcommand{\Ss}{\mathbb{S}}

\newcommand{\SAT}{\mathbb{SAT}}

\newcommand{\uu}{{\bf u}}
\newcommand{\g}{{\bf g}}

\newcommand{\vv}{{\bf v}}
\newcommand{\f}{{\bf f}}

\begin{document}

\title[Summation-by-Parts]{ Review of Summation-by-parts schemes for initial-boundary-value problems}\author{Magnus Sv\"ard  and Jan Nordstr{\"o}m }


\date{\today}
\maketitle
\begin{abstract}
High-order finite difference methods are efficient, easy to program, scales well in multiple dimensions and can be modified locally for various reasons (such as shock treatment for example). The main drawback have been the complicated and sometimes even mysterious stability treatment at boundaries and interfaces required for a stable scheme. The research on summation-by-parts operators and weak boundary conditions during the last 20 years have removed this drawback and now reached a mature state. It is now possible to construct stable and high order accurate multi-block finite difference schemes in a systematic  building-block-like manner. In this paper we will review this development, point out the main contributions and speculate about the next lines of research in this area.
\end{abstract}
{\bf Keywords:} well posed problems, energy estimates, finite difference; finite volume; boundary conditions; interface conditions; stability; high order of accuracy \\

\section{Introduction}\label{sec:intro}





The research on Summation-By-Parts (SBP) schemes was originally driven by applications in flow problems, including turbulence and wave propagation. The objective was to use highly accurate schemes to allow waves and other small features to travel long distances, or persist for long times. One of the ground-breaking papers showing the benefit of high-order finite-difference methods for  wave propagation problems is \cite{KreissOliger72}. However, it has until recently proven difficult to show the same benefit in realistic simulations. Although it is easy to derive high-order finite difference schemes in the interior of the domain it is non-trivial to find accurate and stable schemes close to boundaries. Furthermore, complicated geometries necessitates multi-block techniques. This poses yet another challenge for high-order finite difference schemes since solutions in different blocks must be glued together in a stable and accurate way. The stencils near boundaries and block interfaces create difficulties. We will focus on the so called Simultaneous-Approximation-Term (SAT) technique where the boundary and interface conditions are imposed weakly.

The fundamental idea of SBP-SAT schemes is to allow proofs of convergence for linear and linearized problems. Convergence proofs form the bedrock of numerical analysis of PDEs since they provide the mathematical foundation that gives credibility to a numerical simulation. Without a proof of convergence, there is no guarantee that the numerical solution has any value at all.  The confidence that a discrete solution is an approximation of the true matematical solution is crucial, not only in practical engineering simulations, but also for the possibility to evaluate the accuracy of the model (i.e. the governing equation) itself and propose improved models. Without this quality assurance it is impossible to distinguish between modeling errors and numerical errors.

The SBP finite difference operators were first derived in \cite{KreissScherer74,KreissScherer77} and approximate coefficients calculated. In \cite{Strand94}, the analysis was revisited and exact expressions for the finite difference coefficients were obtained. However, the SBP finite difference operators alone, only admits stability proofs for very simple problems, and the use was limited. This changed with \cite{CarpenterGottlieb94} when Simultaneous Approximation Terms were proposed to augment SBP schemes. These are penalty-like terms that enforces boundary conditions weakly. With both SBP operators and the SAT technique at hand, stability proofs for more complicated systems of partial differential equations (PDEs) were within reach. 

Finite difference methods are by no means the only choice of high-order schemes. There are numerous other high-order methods with different strengths and weaknesses. However, finite difference schemes are often favored in cases where curvilinear multi-block grids can be generated, due to simpler coding and more efficient use of computer resources. For aerodynamic applications where most of the surface of the aircraft is smooth, this methodology is especially suitable since \emph{i)} curvilinear grids can be generated and \emph{ii)}  the resolution of large normal-to-surface gradients force the use of structured grids anyway.  For very complicated geometries (such as close to landing gears), one can use hybrid methods (a combination of high-order finite differences and an unstructured method) as will be discussed below. Hybrid methods are also preferable in situations where waves propagate in free space after beeing generated by complicated geometrical features.


In this article, we will review the progress made towards towards stable high-order finite difference schemes for fluid dynamics as well as other applications. To this end, we will briefly explain the basic principles in a few examples. We will also discuss the SBP-SAT interpretation of other schemes and recent extensions of SBP-SAT schemes for time integration, non-linear theory and shock capturing.

The article is organized as follows. In Section \ref{sbp_theory}, we present the theory for linear initial-boundary-value problem. We introduce the SBP-SAT concepts via a number of examples in Sections \ref{sec:advec}, \ref{sec:advec_diff} and \ref{sec:advec_int}. In Section \ref{sec:conv_rate} we discuss convergence rates and in Section \ref{sec:non_sat} alternative ways to impose boundary conditions. In Section \ref{sec:system} and \ref{sec:system_disc} we explain the SBP-SAT method in a 2-D example. In Section \ref{sec:strict_stab} we discuss aspects of the time evolution of the discrete system and in Section \ref{sec:dual} we review results regarding dual consistency. In Section \ref{sec:other_schemes} we relate the SBP theory for finite difference schemes to other numerical methods. Section \ref{sec:app} contains a review of the various applications where SBP-SAT schemes have been used to obtain numerical approximations. Finally, we discuss some aspects of non-linear theory in Section \ref{sec:non_lin}.

\section{Theory for Initial Boundary Value Problems}\label{sbp_theory}
We begin by reviewing the general theory for Initial-Boundary-Value Problems (IBVP). Most of the material in this section can be found in \cite{GustafssonKreissOliger}. This sets the scene for the subsequent sections focusing on SBP-SAT schemes.
\subsection{Preliminaries}\label{sec:prels}
 Consider the initial-boundary-value problem
\begin{align}
u_t&=P(x,t,\partial_x)u +F, \quad 0\leq x \leq 1,\quad t\geq 0,\nonumber \\
u(x,0)&=f(x),\label{pde1} \\
L_0(t,\partial_x)u(0,t)&=g_0(t),\nonumber \\
L_0(t,\partial_x)u(1,t)&=g_1(t),\nonumber 
\end{align}
where $u=(u^1, ..,u^m)^T$ and  $P$ is a differential operator with smooth matrix coefficients. $L_0$ and $L_1$ are differential operators defining the boundary conditions. $F=F(x,t)$ is a forcing function.
\begin{definition}\label{def_wp}
The IBVP (\ref{pde1}) with $F=g_0=g_1=0$ is well-posed, if for every $f\in C^{\infty}$ that vanishes in a neighborhood of $x=0,1$, it has a unique smooth solution that satisfies the estimate 
\begin{align}
\|u(\cdot,t)\|\leq Ke^{\alpha_c t}\|f\|\label{est1}
\end{align}
where $K,\alpha_c$ are constants independent of $f$.
\end{definition}
A problem is well-posed if it satisfies an estimate like (\ref{est1}). This require that appropriate boundary conditions are used which, along with the estimate, guarantees that a unique smooth solution exists. 
The extension to inhomogeneous boundary condition is possible via a transformation $\tilde u=u-\Psi$ where $\Psi(x,0)=f(x)$ and $\Psi(\{0,1\},t)=g_{0,1}$ such that $\tilde u$ satisfies (\ref{pde1}) with homogeneous data (and a different but smooth forcing function). However, to obtain $\Psi$, $g_{0,1}$ is required to be differentiable in time. This requirement is not necessary if the problem is strongly well-posed as defined below.
\begin{definition}
The IBVP (\ref{pde1}) is \emph{strongly well-posed}, if it is well-posed and 
\begin{align}
\|u(\cdot,t)\|^2\leq K(t)\left(\|f\|^2+\int_0^t(\|F(\cdot,\tau)\|^2+|g_0(\tau)|^2+|g_1(\tau)|)d\tau\right) \label{est2}
\end{align}
holds instead of (\ref{est1}). The function $K(t)$ is bounded for every finite time and independent of  $F, g_0, g_1, f$.
\end{definition}
\begin{remark}
We have tacitly assumed that boundary and initial data are compatible. Compatibility is necessary to ensure a smooth solution. E.g., for continuity one must require that $g_0(t)=f(0)$ and $g_1(t)=f(1)$. For higher continuity, derivatives of $g_{0,1}$ and $f$ must be related via the equation. (See \cite{GustafssonKreissOliger}.)
\end{remark}


Next, we turn to semi-discretizations of (\ref{pde1}).  Let $x_j=jh$, $j=0,1,...,N$ where $h=1/N$ is the grid spacing.  We define the grid functions  $f_j=f(x_j)$ and $F_j(t)=F(x_j,t)$. With each grid point, we also associate a function (the approximate solution) $v_j(t)$. We will use the notion \emph{smooth grid function} to denote a grid function being the projection of a smooth function. Furthermore, we use $\|\cdot\|_h$ to denote a discrete $L^2$-equivalent norm.

Then we approximate (\ref{pde1}) by
\begin{align}
(v_j)_t&=Q_j(x_j,t)v_j +F_j+\Ss_j, \quad j=0,...,N,\quad t\geq 0\label{scheme1} \\
v_j(0)&=f_j\nonumber 
\end{align}
where $Q_j$ is the approximation of  $P$ at $x_j$. $\Ss_j$ is the SAT term. $\Ss_j$ is zero except at a few points close to the boundary. The next definition is in analogy with Definition \ref{def_wp} above.
\begin{definition}
Consider (\ref{scheme1}) with $F=g_0=g_1=0$. Let $f$ be the projection of a $C^{\infty}$ function that vanish at the boundaries. The approximation is \emph{stable}, if, for all $h\leq h_0$
\begin{align}
\|v(t)\|_h\leq Ke^{\alpha_d t}\|f\|_h\label{disc_est1}
\end{align}
holds and $K,\alpha_d$ are constants independent of $f$.
\end{definition}
The same arguments as in the case of well-posedness can be used here to extend this notion to general inhomogeneous data in $L^2$. To do so, $g_{0}$ and $g_1$ must be differentiable. The  following definition rids this assumption.
\begin{definition}
The approximation (\ref{scheme1}) is \emph{strongly stable} if it is stable and
\begin{align}
\|v(t)\|^2_h\leq K(t)\left(\|f\|^2_2+\max_{\tau\in [0,t]} \|F(\tau)\|_h^2 +\max_{\tau\in [0,t]} \|g_0(\tau)\|_h^2+ \max_{\tau\in [0,t]} \|g_1(\tau)\|_h^2\right)\label{disc_est2}
\end{align}
holds instead of (\ref{disc_est1}). $K(t)$ is bounded for any finite $t$ and independent of $F, g_0, g_1, f$.
\end{definition}
Let $\bar u$ be the projection of the exact solution $u(x,t)$ on the grid, i.e., $\bar u_i=u(x_i,t)$. Then the (local) truncation error, $T$, is defined by
\begin{align}
(\bar u_j)_t&=Q(x_j,t,D)\bar u_j +F_j+ \Ss_j+T_j, \quad j=0,...,N,\quad t\geq 0\nonumber 
\end{align}
For an SBP-SAT scheme the truncation error takes the form,
\begin{align}
T=\left( \OO(h^r), ...\OO(h^r),\OO(h^p)...\OO(h^p),\OO(h^r)...\OO(h^r)\right)^T \label{truncation}
\end{align}
where $r<p$ and the number of points with accuracy $r$ is finite and confined to the vicinity of the boundary. A scheme with $r,p>0$ is termed \emph{consistent}. The \emph{order of accuracy} refers to the exponent in the truncation error. In (\ref{truncation}) that would be $r$  (if $r<p$) or  in this case we may say $(p,r)$ given that the structure of $T$ is known.   Moreover, the (solution) \emph{error} is defined as $e_j(t)=v_j-\bar u_j$ and the \emph{convergence rate} of the method is $q$ if  $\|e\|_2\leq \OO(h^q)$, where $\|\cdot\|_2$ denotes the discrete $L^2$ norm.

Although the definitions of (strong) well-posedness and (strong) stability are similar, the bounds in the corresponding estimates need not be the same. The following definition connects the growth rates of the continuous and semi-discrete solutions.
\begin{definition}\label{def:strict_stab}
Assume that (\ref{pde1}) is well-posed with $\alpha_c$ in (\ref{est1}) and that the semi-discrete approximation is stable with $\alpha_d$ in (\ref{disc_est1}). If $\alpha_d\leq \alpha_c+\mathcal{O}(h)$ for $h\leq h_0$ we say that the approximation is \emph{strictly stable}.
\end{definition}

In the above definitions, the schemes are semi-discrete, i.e., time is left continuous. Clearly, only fully discrete schemes are useful in practice. In \cite{KreissWu93}, it was shown that semi-discrete stable schemes are, under certain conditions, stable when discretized in time using Runge-Kutta schemes. This problem was further studied in \cite{LevyTadmor98} with a particular focus on energy stable semi-discrete schemes (such as the scheme above). Recently it was shown in \cite{NordLun13} how to extend the the SBP-SAT technique in space to the time-domain, where fully discrete sharp energy estimates can be obtained, see section \ref{sec:compaspects} for more details.

\section{Theory for SBP-SAT schemes}\label{sbp_sat_theory}
In this section we present the SBP-SAT schemes by considering a few examples. 

\subsection{The advection equation}\label{sec:advec}

Consider 
\begin{align}
u_t+au_x&=0,\quad 0\leq x \leq 1, \quad t>0,\label{const_advec}\\
u(x,0)&=f(x),\nonumber \\
u(0,t)&=g_0(t)\nonumber
\end{align}
where $f$ and $g_0$ are  initial and boundary data in $L^2$ and $a>0$. The semi-discretization of (\ref{const_advec}) is,
\begin{align}
\uu_t+aD\uu&=P^{-1}\Ss,\quad  \quad t>0 \label{disc_advec}\\
\uu(0)&=\f \nonumber 
\end{align}
where $\uu(t)=(u_0(t),u_1(t),...u_N(t))^T$ and similarly $\f=(f(x_0),...,f(x_N))^T$. $\Ss=(\Ss_0,0,...,0)^T$ is the SAT enforcing the boundary condition (to be defined later).  $D$ is a difference matrix.
\begin{definition}\label{def_sbp}
$D$ is a $(p,r)$-accurate first-derivative SBP operator, if its truncation error is given by (\ref{truncation}), $D=P^{-1}Q$ where $Q+Q^T=B=diag(-1,0,0,...,0,1)$, $P$ is symmetric positive definite and defines an $L^2$-equivalent norm $\|v\|^2_P=v^TPv$.
\end{definition}
The SBP operators for first-derivative approximations associated with explicit central difference schemes are found in \cite{Strand94}. (Note that the entries of $P$ scale as the stepsize h.) In  \cite{KreissScherer74} it was proven that if $P$ is a diagonal matrix, the boundary closures can be at most  of order $r=\tau$ when the interior stencils are of order $p=2\tau$. Such operators exist with order up to eight in the interior. If $P$ is allowed to be non-diagonal with small blocks at the boundaries, the truncation error can be $(\tau,\tau-1)$.  Furthermore, there exists SBP schemes for other than explicit finite differences. In \cite{AbarbanelChertock00,Abarbanel_etal00} and \cite{CarpenterGottlieb94}, SBP-type boundary closures for compact interior discretizations were derived. 

Let $a^+=\max(a,0)=(|a|+a)/2$. Stability of the advection equation was established in \cite{CarpenterGottlieb94} for SBP-SAT schemes.
\begin{proposition}\label{prop_advec}
Let $D$ be an SBP operator and $\Ss_0=\sigma a^+[P^{-1}]_0(u_0-g_0)$. If $\sigma<-1/2$ the scheme (\ref{disc_advec}) is strongly stable.
\end{proposition}
\begin{proof}
$\|\uu\|^2_t=\uu_t^TP\uu+\uu P \uu_t=-a\uu^T(Q+Q^T)\uu+2\uu \Ss$ gives $\|\uu\|^2_t\leq au^2_0+2a\sigma u_0(u_0-g_0)$. If $\sigma<-1/2$ we obtain $\|\uu\|^2_t\leq -\frac{a\sigma^2}{(1+2\sigma)} g_0^2$. 
\end{proof}

\subsection{The advection-diffusion equation}\label{sec:advec_diff}
Consider
\begin{align}
u_t+au_x&=(\epsilon u_x)_x,\quad 0\leq x \leq 1, \quad t>0\label{advec_diff}.\\
u(x,0)&=f(x)\nonumber \\
au(0,t)-\epsilon u_x(0,t)&=g_0(t)\nonumber\\
\epsilon u_x(1,t)&=g_1(t)\nonumber 
\end{align}
where both $a$ and $\epsilon$ are positive constants. 
\begin{remark}
Equation (\ref{advec_diff}) is often used as a model problem for Navier-Stokes equations and with that interpretation,  the left and right boundary conditions are prototypes of far-field conditions based on the characteristic direction $a$. 
\end{remark}

To demonstrate that the problem is strongly well-posed, we apply the energy method.
\begin{align}
\|u\|_t^2+2\epsilon \|u_x\|^2&=(au^2-2\epsilon u u_x)_0-(au^2-2\epsilon u u_x)_1 \nonumber \\
                                                &= a^{-1}\left[ (au - \epsilon u_x)^2 - (\epsilon u_x)^2\right]_0-
                                                      a^{-1}\left[ (au - \epsilon u_x)^2 - (\epsilon u_x)^2\right]_1. \nonumber 
\end{align}
The construction above where the boundary terms are factored into clean squares was first done in
\cite{Nordstrom95(1)} and later used in \cite{NordstromCarpenter99}. Keeping in mind that $a$ is positive
we obtain the final strong estimate
\begin{equation}\label{adv_diff_est} 
\|u\|_t^2+2\epsilon \|u_x\|^2= a^{-1}\left[g_0^2 - (\epsilon u_x)^2_0 \right]-a^{-1}\left[(au - \epsilon u_x)^2_1 - g_1^2 \right].
\end{equation}
The most straightforward SBP-SAT discretization of this problem is
\begin{align}
\uu_t+aD\uu=D(\epsilon D\uu) + P^{-1}\Ss\label{disc_advec_diff},
\end{align}
where
\[ (\Ss_0)_0=\sigma_0(au_0-\epsilon (D\uu)_0-g_0(t)),\quad (\Ss_1)_N=\sigma_1(\epsilon (D\uu)_N-g_1(t)).
\]
\begin{proposition}
The scheme is strongly stable with $\sigma_0=-1,\sigma_1=-1$.
\end{proposition}
\begin{proof}
Use the energy method to obtain
\begin{align}
\|\uu\|_t^2+2\epsilon \|D\uu\|^2&=a^{-1}\left[g_0^2 - (au-g_0)^2_0 \right]-a^{-1}\left[(au -g_1)^2_N - g_1^2 \right],\nonumber 
\end{align}
which is very similar to the continuous estimate (\ref{adv_diff_est}) by considering the boundary conditions. \end{proof}
It should be remarked that the left boundary conditions is easily generalized to
\begin{align}
\beta au(0,t)-\epsilon u_x(0,t)&=g_0(t),\quad \beta>\frac{1}{2}.\nonumber
\end{align}

The advection-diffusion equation in the SBP-SAT framework was first studied in \cite{CarpenterNordstrom99}, where the original proofs are found. Furthermore, the authors studied discretizations where the domain is subdivided in blocks, with possibly different grid spacings and the stability conditions at the interfaces were derived.

We close this section by pointing out that other boundary conditions than those in (\ref{advec_diff}) are possible. For instance in \cite{SvardNordstrom08}, Dirichlet boundary conditions were proven stable.

\subsection{Interface treatment for the advection equation}\label{sec:advec_int}

To be able to generate a grid, it is often necessary to split the domain into several patches, each containing one smooth piece of the grid. (See the Section \ref{sec:intro} and \ref{sec:multiD_NS} for discussions on multi-block gridding.)


Here, we will demonstrate the interface procedure for the advection equation. Consider,
\begin{align}
u_t+au_x&=0,\quad -1\leq x \leq 1, \quad t>0\label{advec_inter}.\\
u(x,0)&=f(x)\nonumber \\
au(-1,t)&=a g_{-1}(t)\nonumber
\end{align}
where $a>0$ is a constant. In Section \ref{sec:advec}, we showed that this problem is well-posed. (The change in domain size is not important for this conclusion.) Furthermore, a stable single-grid discretization, including the boundary condition, was  proposed. Here, we split the domain at $x=0$. We discretize the left piece, $-1\leq x \leq 0$ with $N+1$ points and the right, $0\leq x \leq 1$, with $M+1$. We denote the grid spacings as $h_L=1/M$ and $h_R=1/N$. The subscripts $L$ and $R$ will be used to signify operators in the two domains, e.g. $D_L=P^{-1}_LQ_L$ is the SBP difference operator in the left domain. We denote the discrete solution vectors as $\vv$ in the left domain and $\uu$ in the right. Note that both $v_N(t)$ and $u_0(t)$ represent approximations of $u(0,t)$. Hence, the interface condition $v_N(t)=u_0(t)$ will be enforced weakly in the scheme. We write the semi-discrete scheme as,
\begin{align}
\vv_t+aD_L\vv&=\Ss_L+\Ss_{0L}, \quad t>0 \nonumber \\
\vv(0)&=\f_L, \label{disc_inter}\\
\uu_t+aD_R\uu&=\Ss_{0R}, \nonumber \\
\uu(0)&=\f_R. \nonumber
\end{align}
Here $\Ss_L$ is the SAT at $x=-1$ and corresponds to $\Ss$ in Proposition \ref{prop_advec}. 
\begin{proposition}
Let $D_L,D_R$ be SBP operators, $[\Ss_{0L}]_N=\sigma_L [P^{-1}_L]_N(v_N-u_0)$ and $[\Ss_{0R}]_0=\sigma_R [P^{-1}_R]_0(u_0-v_N)$ with $\sigma_L\leq \frac{a}{2}$ and $\sigma_R=\sigma_L-a$. Then the scheme (\ref{disc_inter}) is strongly stable. Furthermore, the discretization is conservative in the sense that it satisfies the weak form of the equation.
\end{proposition}

\begin{proof}
In the energy estimate we lump all the boundary terms at $x=-1$ in a term $BT_L$ and conclude that they are bounded due to Proposition \ref{prop_advec}. The energy estimate is obtained by multiplying the left scheme by $\vv^TP_L$ and the right by $\uu^TP_R$ and adding the results,
\begin{align}
(\|\vv\|^2)_t+(\|\uu\|^2)_t=-av_N^2+au_0^2+2v_N\sigma_L (v_N-u_0)+2u_0\sigma_R (u_0-v_N)+BT_L.\nonumber 
\end{align}
Using $\sigma_L\leq a/2$ and $\sigma_R=\sigma_L-a$ yields negative coupling terms. 
Hence, $(\|\vv\|^2)_t+(\|\uu\|^2)_t\leq BT_L$ and we have a bound. 

\begin{remark}
Note that the sign of $a$ is immaterial to the interface part of this proof. The SAT terms at the interface are identical if $a<0$
\end{remark}
Next, we turn to conservation. Specifically, we show that the discrete scheme will, upon convergence, capture a weak solution. To this end, we introduce a smooth test function $\Phi$ that vanishes at the boundaries. The weak form of (\ref{advec_inter}) is
\begin{align}
\int_{-1}^1\Phi u\,dx|_0^t-\int_0^t\int_{-1}^1   (\Phi_tu+\Phi_xau)\,dxdt=0\label{weak_lin}
\end{align}
We denote its restriction to the grid as $\phi_{L,R}$ such that $(\phi_{L,R})_j(t)=\Phi(x_j,t)$ and $(\phi_L)_N=(\phi_R)_0$. We multiply the left part of the scheme (\ref{disc_inter}) by $\phi^TP_L$, the right by $\phi^TP_R$ and integrate in time to obtain,
\begin{align}
&\phi_L^TP_L\vv|_0^T+\phi_R^TP_R\uu|_0^T - \nonumber \\
&\int_0^T(\phi_L)^T_tP_L\uu+(\phi_R)_t^TP_R\uu +a(D_L\phi_L)^TP_L\vv
+a(D_R\phi_R)^TP_R\uu\,dt=0 \label{weak_lin_disc}
\end{align}
where all the terms at the interface cancel out due to the conservation condition $\sigma_R=\sigma_L-a$. The relation (\ref{weak_lin_disc}) converges to (\ref{weak_lin}) thanks to the $L^2$ bound on the discrete solution.
\end{proof}

The proof here is a special case of the one given in \cite{CarpenterNordstrom99}. (See also \cite{GongNordstrom11,BergNordstrom12(2),ErikssonAbbasNordstrom11}.) In \cite{CarpenterNordstrom10} a general treatment of interfaces is presented. Non-conforming interface procedures are derived in \cite{NissenKreiss12,MattssonCarpenter10}. For generalizations to the multi-D Euler and Navier-Stokes equations, we refer to \cite{NordstromCarpenter99,NordstromGong09}. Non-linear analysis of interfaces can be found found in \cite{FisherCarpenter12,SvardOzcan13}

\subsubsection{Second-derivative SBP operators}\label{sec:sec_deriv}

In (\ref{disc_advec_diff}) the second derivative is discretized by applying the first derivative approximation twice. This leads to a wide stencil for the second derivative approximation.  To increase the effective accuracy, more compact stencils approximating the second derivative can be used.  That is (assuming $\epsilon=constant$), we would like a discretization 
\begin{align}
\uu_t+aD\uu=\epsilon D_2\uu + \Ss\label{disc_advec_diff2}
\end{align}
For this approximation to be SBP, certain conditions on $D_2$ must be imposed. They were originally identified in \cite{CarpenterNordstrom99}.
\begin{definition}
The second-derivative SBP matrix is defined as $D_2=P^{-1}(-S^TM+B)S$ with the following properties. Let $\uu$ be a smooth grid function, then $D_2\uu=u_{xx}+\mathcal{O}(h^r)$ and $[S\uu]_{0,N}=u_x(x_{0,N})+\mathcal{O}(h^r)$. Furthermore, $M$ is symmetric positive definite with $h n I \leq M \leq Nh I$ where $n,N$ are some constants.
\end{definition}
Clearly, $D_2=DD=P^{-1}QP^{-1}Q=P^{-1}((P^{-1}Q)^TP-B)P^{-1}Q$ is one such derivative with $S=P^{-1}Q$ and $M=P$.  Others are found in \cite{CarpenterNordstrom99,MattssonNordstrom04}. The $D_2$ operators have similar accuracy constraints as the first-derivative approximations. That is, with diagonal $P$ the order is $(2\tau,\tau)$ and with non-diagonal $P$, the boundary accuracy can be raised to ($\tau,\tau-1$).

To accommodate a varying $\epsilon$ while still keeping the stencil narrow, a new set of operators were derived in \cite{Mattsson11}.  Furthermore, in \cite{MattssonSvard08} the accuracy and stability properties of various different SBP second-derivative approximations were discussed.

The papers  \cite{AbarbanelDitkowski97,AbarbanelDitkowski99}  are also SBP-SAT-like schemes for diffusive equations, but designed as an immersed boundary technique. That is, the domain is covered by a non-conforming Cartesian grid.

\begin{remark}
An often expressed criticism against the use of $DD\uu$ is that the stencils are wide and that they do not damp the highest frequency and therefore are prone to "spurious oscillations''. That is only true for non-linear problem in the non-linear regime for which the stability proofs do not hold. For linear problems such oscillations can not be generated by a stable scheme due to the bound in $L^2\cap L^{\infty}$.
\end{remark}

\subsection{Convergence rates}\label{sec:conv_rate}

The famous Lax-Richtmyer Theorem (\cite{LaxRichtmyer56}) states that a consistent scheme is convergent if it is stable. The theorem is very generally posed and does not require data (and consequently the solution) to be smooth. Therefore, no prediction of the convergence rate is possible. Our definitions of well-posedness requires solutions to be smooth, which in turn allows estimates of convergence rates. However, deriving convergence rates is a non-trivial task. Considering, (\ref{pde1}) and its discretization (\ref{scheme1}), it is easy to see that the error $e$ satisfies a scheme equivalent to (\ref{scheme1}) and the estimates (\ref{disc_est1}) or (\ref{disc_est2}) apply to $e$. Therefore, if $r=p$, the convergence rate is trivially $p$. If on the other hand $r<p$, and since the estimates are stated in $L^2$, the convergence rate can be shown to be at least $\min(r+1/2,p)$. This rate, however, is suboptimal.

In \cite{Gustafsson75} and \cite{Gustafsson81} the question of global convergence rate was studied. These papers are commonly cited to motivate a drop in the boundary accuracy according to the rule: "For a $(p,r)$-accurate approximation with $r<p$, the convergence rate is $r+1$". This is true given that the scheme satisfies a determinant condition (or equivalently the Kreiss condition \cite{GustafssonKreissOliger}). However, it is not easy to prove that the Kreiss condition is satisfied for a realistic problem, and hence the analysis is often neglected. It would be desirable to infer the higher convergence rate at the boundary directly from the energy estimate. An effort in this direction is found, in \cite{AbarbanelDitkowski00} where parabolic problems were studied and using the energy method the rate $\min(r+3/2,p)$ was proven. In practice, however, a convergence rate of $\min(r+2,p)$ is observed.

In \cite{SvardNordstrom06} the question of convergence rate was revisited with the intention to tie the convergence rate to energy estimates. For schemes approximating PDEs with principal part of order $q$ and under certain accuracy conditions for the boundary conditions, it was shown that the convergence rate is $\min(r+q,p)$, if the numerical solution is bounded in $L^2\cap L^{\infty}$. In a recent paper \cite{Svard12_3},  $L^{\infty}$ bounds were derived using the energy estimates. Hence, any scheme satisfying an energy estimate is automatically bounded in $L^{\infty}$.

In (\ref{disc_advec_diff}), we approximate the second derivative by applying $D$ twice. This makes the approximation of the second-derivative of order $(r-1,p)$, see \cite{SvardNordstrom06}.  Thanks to the bound in ${L^2\cap L^{\infty}}$ (\cite{Svard12_3}) and \cite{SvardNordstrom06}, the global convergence rate is $\min(r+1,p)$. The increased truncation error is exactly matched by the increased convergence rate for a parabolic problem. However, with the dedicated second-derivative approximations (\ref{disc_advec_diff2}), the truncation error is $(r,p)$ for both the hyperbolic and parabolic part and consequently the convergence rate is $\min(r+2,p)$. For precise accuracy conditions on each operator, we refer to \cite{SvardNordstrom06}.

\begin{remark}
Note that the results above hold even in the extreme case where the approximation of the PDE is inconsistent near the boundary. For instance, the biharmonic equation $u_t=-u_{xxxx}$ was computed by applying the (4,3)-accurate first-derivative 4 times. This rendered the approximation of $u_{xxxx}$ to be 0th-order accurate near the boundary. Still, a convergence rate of 4 was recorded in accordance with the theory.
\end{remark}


\subsection{Alternative boundary treatments}\label{sec:non_sat}

The SAT technique to enforce boundary conditions has been instrumental in the development of high-order finite difference schemes. It is relatively simple to implement and very versatile when proving stability. However, SATs are not the only way to enforce boundary conditions in conjunction with the SBP difference operators. The simplest, and most common, way to enforce boundary conditions for finite difference schemes is simply to overwrite the solution values on the boundary with boundary data every time step. We refer to this technique as the \emph{injection method}. For simple PDEs, such as the advection equation, this may yield a stable approximation. However, already for linear systems in 1-D with waves travelling through the boundary in both directions this technique is doubtful. In \cite{SvardMishra12} it was noted that injection was unstable for the Euler equations even at a supersonic inflow (where it would appear sensible to use since all characteristics are directed into the domain). For more information on the injection method, see \cite{GustafssonKreissOliger}.

Another technique to enforce boundary conditions is the projection method. (See \cite{Olsson95(1),Olsson95(2), GerritsenOlsson96, GerritsenOlsson98}.) In this technique the solution is projected in each time step onto a subspace satisfying the boundary conditions and stability is proven with the energy method. For an extensive evaluation of the different techniques to enforce boundary conditions, see \cite{Mattsson03}. See also \cite{Bodony10} for a thorough evaluation of SAT boundary conditions in aeroacoustic applications; in \cite{BergNordstrom11} and also \cite{SvardNordstrom08} further  accuracy assessments of the SAT procedure are found.

Furthermore, the SAT method is a weak implementation of the boundary condition in the sense that the boundary points ($\uu_{\{0,N\}}$) are unknowns and never explicitly set to the boundary value. Indeed, the boundary conditions are generally never satisfied exactly. The SAT acts as a forcing term driving this discrepancy to zero. On the other hand, both injection and projection are examples of strong enforcement of boundary conditions. Strong enforcement can sometimes be used for linear problems, if stability can be obtained. However, for the Euler equations, it is more subtle. It can be shown that to obtain a well-posed problem, the boundary conditions cannot be enforced strongly. (See \cite{DuBoisLeFloch88}). This makes the SAT technique the preferred choice for non-linear problems.

\subsection{Systems}\label{sec:system}

We will introduce the use of SBP-SAT schemes in a multi-dimensional setting for a hyperbolic system of partial differential equations, which serves as a model for the 2-D Euler equations. For simplicity, we assume that the boundaries are all of far-field type.
\begin{align}
u_t+Au_x+Bu_y&=0, \quad 0<x,y<1, \quad 0< t\leq T \label{system}\\
A^+u(0,y,t)&=A^+g_W(y,t)\nonumber \\
A^-u(1,y,t)&=A^-g_E(y,t)\nonumber\\ 
B^+u(x,0,t)&=B^+g_S(x,t)\nonumber \\
B^-u(x,1,t)&=B^-g_N(x,t)\nonumber \\
u(x,y,0)&=f(x,y)\nonumber 
\end{align}
Here, $u$ is a column vector with $m$ components and the matrices $A,B$ are symmetric. Let $A^{\pm},B^{\pm}$ denote their positive and negative parts, i.e. $X\Lambda^+ X^T=A^+ $ where $\Lambda^+$ contains the positive eigenvalues and $\Lambda^+ + \Lambda^- = \Lambda$.   We assume that $g_{W,E,S,N},f$ are bounded in $L^2\cap L^{\infty}$ and $\|\cdot\|$ denotes the $L^2$-norm.  The number of boundary conditions are correct since only in-going characteristics are given at ${x,y}={0, 1}$.

First, we demonstrate well-posedness of (\ref{system}) by deriving an energy estimate. That is we multiply (\ref{system}) by $u^T$ and integrate by parts in space.
\begin{align}
\frac{1}{2}\|u\|_t^2+\int_0^{1}\int_0^1u^TAu_x dx\,dy +\int_0^{1}\int_0^1u^TBu_y dx\,dy =0\nonumber 
\end{align}
or,
\begin{align}
\|u\|_t^2-\int_0^{1}u(0,y,t)^TAu(0,y,t) dy +\int_0^{1}u(1,y,t)^TAu(1,y,t) dy& \nonumber \\
-\int_0^{1}u(x,0,t)^TBu(x,0,t) dx +\int_0^{1}u(x,1,t)^TBu(x,1,t) dx &=0\nonumber 
\end{align}
In analogy with the SATs in the discrete case, we invoke the boundary conditions by adding
\begin{align}
+2 \int_0^{1}u(0,y,t)^T(A^+u(0,y,t)-A^+g_W(y,t)) dy\quad (=0) \nonumber \\
-2 \int_0^{1}u(1,y,t)^T(A^-u(1,y,t)-A^-g_E(y,t)) dy\quad (=0) \label{weakBC}\\
+2 \int_0^{1}u(x,0,t)^T(B^+u(x,0,t)-B^+g_S(x,t)) dy\quad (=0) \nonumber \\
-2 \int_0^{1}u(x,1,t)^T(B^-u(x,1,t)-B^-g_N(x,t)) dy\quad (=0) \nonumber 
\end{align}
\begin{remark}
We could have added a scaling parameter in front of every boundary term in (\ref{weakBC}). The derivation below would then give a range of permissible values of the parameters.
\end{remark}
To reduce notation, we only write down the derivation for $x=0$ terms.
\begin{align}
\|u\|_t^2-\int_0^{1}u(0,y,t)^TA^-u(0,y,t) dy +\int_0^{1}u(0,y,t)^TA^+u(0,y,t) dy &\nonumber \\
-2\int_0^{1}u(0,y,t)^TA^+g_W(y,t) dy + ... &=0\nonumber 
\end{align}
The terms associated with outflow at each boundary do not contribute to a growth and are dropped. (The $A^-$ term at $x=0$ etc.) We obtain
\begin{align}
\|u\|_t^2&\leq \int_0^{1} g_W^TA^+g_W\,dy-\int_0^{1} (u-g_W)^TA^+(u-g_W)\,dy. \label{contest}
\end{align}
The first term on the right-hand side is positive but bounded and represents the energy that is passed into the domain through the boundary. The second term vanishes thanks to the boundary condition.  Integrating in time, from 0 to the final time $T$, will give the desired bound on $\|u(\cdot,\cdot,T)\|$. 
\begin{remark}
No-penetration wall boundary conditions  can be treated in a similar way and an energy estimates will follow, \cite{SvardNordstrom08}.
\end{remark}
\begin{remark}
If $A$ and $B$ are obtained from the linearized Euler equations, the above estimate demonstrates stability of smooth solutions to the non-linear Euler equations.
\end{remark}

\subsection{Spatial discretization}\label{sec:system_disc}

To define an SBP-SAT semi-discretization of (\ref{system}), we introduce the computational grid, $x_i=ih_x$, $i\in\{0,1,2,...,N\}$ and $y_j=jh_y$, $j\in\{0,1,2,...,M\}$ where $h_x,h_y>0$ are the grid spacings.   We will also need the vectors $e_0=(1,0,0...,0)^T$ and $e_N=(0,...,0,1)^T$, and the matrices $E_0=e_0e_0^T$ (1 in the upper-right corner and 0 elsewhere) and $E_N=e_Ne_N^T$ (1 in the lower-left corner and 0 elsewhere).   With this notation we have $Q+Q^T=-E_0+E_N$. 

Next we extend the SBP operators to two dimensions by the use of Kronecker products. The Kronecker product of two matrices $A,B$ is defined as
\begin{align}
A\otimes B = \left(\begin{array}{ccc}
a_{11}B & \hdots & a_{1N}B\\
\vdots & &\vdots\\
a_{N1}B & \hdots & a_{NN}B\\
\end{array} \right).
\end{align}
It satisfies $A\otimes B + C\otimes D = (A+C)\otimes (B+D)$ and $(A\otimes B)(C\otimes D) = (AC\otimes BD)$ assuming that both $A$ and $C$ have the same dimensions and so has $B$ and $D$. Furthermore, $(A\otimes B)^T=(A^T\otimes B^T)$ and, if $A,B$ are invertible then $(A\otimes B)^{-1}=(A^{-1}\otimes B^{-1})$. 

Let $I_{M} $ denote an $M$-by-$M$ identity matrix and let $v_{kij}$ denote the numerical approximation of $u_k(x_i,y_j)$, i.e., the kth component of $u$ at $x_i,y_j$.  We introduce:
\begin{align}
\vv &= (v_{111}, ...v_{m11}, v_{121}...v_{m21},...v_{mN1},v_{112}, ...v_{m12},... )^T \nonumber \\
\Pp&=(P_y\otimes P_x\otimes I_m) &\nonumber 
\end{align}
$\Pp$ defines a 2D $l^2$-norm by $\|\vv\|^2=\vv^T \Pp \vv$. Note that $(I_M\otimes D_N\otimes I_m)\vv$ calculates the $x$-derivative approximation in the entire domain. Also, $(I_M\otimes E_0 \otimes I_m)\vv$ extracts the boundary points at $x=0$, i.e., it sets all but the boundary points to 0 in $\vv$. With these observations, we write the scheme as
\begin{align}
\vv_t + (I_M \otimes D_x \otimes A)\vv + (D_M \otimes I_x \otimes B)\vv =\SAT\label{scheme}
\end{align}
where
\begin{align}
\SAT=&-(I_M\otimes P^{-1}_xE_0 \otimes A^+)(\vv-\g)\nonumber \\
&+(I_M\otimes P^{-1}_xE_N \otimes A^-)(\vv-\g)\label{SAT}\\
&-(P_y^{-1}E_0\otimes I_N \otimes B^+)(\vv-\g)\nonumber \\
&+(P_y^{-1}E_M\otimes I_N \otimes B^-)(\vv-\g)\nonumber \\
\end{align}
$\g$ has the same structure as $\vv$ with $g_{W,E,S,N}$ injected on the appropriate boundary positions and it is 0 elsewhere. 

To reduce notation below, we single out one boundary
\begin{align}
\SAT = -(I_M\otimes P^{-1}_xE_0 \otimes A^+)(\vv-\g) +SAT\label{SAT2}
\end{align}
Next, we derive a bound on $\vv$. That is,  we multiply (\ref{scheme}) from the left by $\vv^T\Pp$.
\begin{align}
\vv^T\Pp\vv_t + \vv^T(P_y \otimes Q_x \otimes A)\vv + \vv^T(Q_y \otimes P_x \otimes B)\vv &=\nonumber \\
-\vv^T(P_y\otimes E_0 \otimes A^+)(\vv-\g) + \vv^T\Pp SAT\nonumber 
\end{align}
(Notice that the first term on the right-hand side corresponds exactly to the first in (\ref{weakBC}).)  Adding the transposed equation, and using $Q_N+Q_N^T=-E_0+E_N$,  yields, 
\begin{align}
(\vv^T\Pp\vv)_t + \vv^T(P_y \otimes -E_0+E_N \otimes A)\vv + \vv^T(-E_0+E_M \otimes I_N \otimes B)\vv &=\nonumber \\
-2\vv^T(P_y\otimes E_0 \otimes A^+)(\vv-\g) + 2\vv^T\Pp SAT\nonumber 
\end{align}
To reduce notation, we set all boundary terms except those at $x=0$ to 0. (The others are dealt with in the same way.) We get,
\begin{align}
\|\vv\|^2_t &=  \vv^T(P_y \otimes E_0 \otimes (A^++A^-)\vv -2\vv^T(P_y\otimes E_0 \otimes A^+)(\vv-\g) \nonumber \\
\|\vv\|^2_t &\leq -\vv^T(P_y \otimes E_0 \otimes A^+)\vv +2\vv(P_y\otimes E_0 \otimes A^+)\g \nonumber 
\end{align}
Denote $M=(P_y \otimes E_0 \otimes A^+)$ and note that the matrix is bounded and positive semi-definite. We obtain
\begin{align}
\|\vv\|^2_t \leq -\vv^TM\vv +2\vv M \g = \g^T M \g - (v-\g)^T M (v-\g), \label{disc_est}
\end{align}
which results in a bounded growth and hence the scheme is stable. The right-hand side (\ref{disc_est}) is a direct analog of (\ref{contest}). The difference here is that $\vv$ is not necessarily equal to $\g$ and hence a small dissipation is added at the boundary. 

The SAT (\ref{SAT2}) could be stated more generally as,
\begin{align}
\SAT=-(I_y\otimes P^{-1}_xE_0 \otimes \tilde A)(\vv-g)+SAT
\end{align}
where we have shown stability for $\tilde A=A^+$ which represents a far-field boundary. By choosing $\tilde A$ differently, we may enforce a no-penetration condition and prove stability in a similar way. Interfaces between different grids also fit in this framework; $\tilde A$ is chosen to obtain an energy estimate and $\g$ would contain data from the neighboring grid block. We stress that an interface is not a boundary in its ordinary meaning. Therefore, there is no constraint with regards to the number of conditions given (although too few, meaning less than in-going characteristics, will lead to an unstable approximation). 

For more information on boundary procedures and generalizations to the Navier-Stokes equations, see \cite {NordstromGong09,SvardNordstrom08,SvardCarpenter07,GongNordstrom11}.

\subsection{Strict stability}\label{sec:strict_stab}

So far we have discussed standard energy estimates which produce $L^2$-bounds on the solution at any final time $T$. Such estimates ensure stability in the sense that on any compact domain (in space and time) the solution converges as the grid is refined. For a particular grid resolution, however, the time evolution of the energy may differ between the numerical and true solution.  In Definition \ref{def:strict_stab}, strict stability was defined to ensure a numerically correct time evolution of the discrete solution; the time evolution of the norm should  converge with $h$ as the grid is refined. Sometimes it is desirable to take this requirement one step further by demanding that $\alpha_{d}=\alpha_{c}-|ch|$ such that the discrete scheme is more dissipative than the PDE itself. For constant coefficient PDEs on a Cartesian domain, this is generally not an issue since strictly stable approximations are usually obtained. For variable coefficient PDEs, this is not the case.

\subsubsection{Splitting}\label{sec:split}

Consider a variable coefficient scalar PDE 
\begin{align}
u_t+ (a(x)u)_x = 0, \quad 0<x<1\label{non-split}
\end{align}
augmented with appropriate boundary and initial conditions. We assume that $a(x)$ is smooth such that $\max_x|a_x|$ is bounded.  For smooth solutions we can rewrite (\ref{non-split}) on skew-symmetric form
\begin{align}
u_t+ \frac{1}{2}(au)_x + \frac{1}{2}au_x +\frac{1}{2}a_xu = 0.\label{split_eqn}
\end{align}
The equations (\ref{non-split}) and (\ref{split_eqn}) are of course equivalent. We assume $a(x)>0$ and give a boundary condition to the left $u(x_L,t)=0$. Applying the energy method results in 
\begin{align}
\|u\|_t^2=-a(x_R)u_R^2-\int_{x_L}^{x_R}a_xu^2 \,dx.\nonumber 
\end{align}
The last term need not be negative since we have not assumed anything about the sign of $a_x$. Hence, the energy may grow. However, it is bounded since, 
\begin{align}
\|u\|^2_t\leq \max_x(|a_x|)\|u\|^2.\label{growth2}
\end{align}
We now have two choices of discretizations. The first is
\begin{align}
u_t+P^{-1}Q(Au)=\SAT.\nonumber 
\end{align}
We omit further details and conclude that it is possible to prove stability by obtaining an energy estimate. (See \cite{MishraSvard10} for estimates of variable coefficient problems.)  If, on the other hand, we choose the scheme
\begin{align}
u_t+ \frac{1}{2}D(Au) + \frac{1}{2}ADu +\frac{1}{2}A_xu = 0 \nonumber 
\end{align}
where the entries of the matrix $A_x$ are computed using $P^{-1}Q{\bf a}$ ($[{\bf a}]_j=a(x_j)$), we obtain the more accurate growth rate corresponding  to (\ref{growth2}). For more details on splitted schemes and strict stability,  see \cite{Nordstrom06,Nordstrom07}. Computational results demonstrating the effect of splitting and diagonal norms ($P$ diagonal) are found in \cite{KozdonDunham12(2), KozdonDunham12} where the linear elastic wave equation is combined with nonlinear interface conditions (friction laws for faults)  in earthquake modeling.

Similar splitting techniques can be exploited to obtain energy-style estimates for non-linear problems. Such splittings are generalizations of the well-known "1/3''-trick for Burgers' equation, $u_t+(u^2/2)_x=0$. For smooth solutions this can be written as
\begin{align}
u_t+\frac{1}{3}u^2_x+\frac{1}{3}uu_x=0,\nonumber
\end{align}
for which an estimate of $\|u\|^2$ is readily obtained. The key is homogeneity of the flux as observed in \cite{OlssonOliger94}, where estimates for the Euler equations were derived. This was further developed in \cite{GerritsenOlsson96} and \cite{SvardMishra12}. A flux-splitting technique was also analyzed as a mean to stabilize non-linear equations in \cite{FisherCarpenter12}.

\subsubsection{Coordinate transformations}\label{sec:coord}

Next, we will discuss the effect of coordinate transformations on the time evolution of the semi-discrete problem. Consider
\begin{align}
u_t+u_x=0, \quad u(0,t)=0, \quad 0\leq x \leq 1\label{advec_stretch}
\end{align}
The energy method gives the following estimate: $\|u\|^2=-u^2(1,t)$. The decay of the energy is precisely the energy disappearing through the right outflow boundary. If we discretize this problem with a constant grid size $h$, we get
\begin{align}
u_t+P^{-1}Qu = -\frac{1}{2}[P^{-1}]_0(u_0-0)\nonumber
\end{align}
which results in the estimate $\|u\|_t^2= -u_N^2$, i.e., exactly the same as the continuous estimate. The discretization is strictly stable. 

Next introduce a (non-singular) coordinate transformation, $x=x(\xi)$ (e.g. a stretched grid). Then, we obtain the equivalent equations
\begin{align}
u_t+u_\xi\xi_x=0, & \quad u(0,t)=0,\label{stretch1} \\
\textrm{or,}&\nonumber \\
(\xi_x)^{-1}u_t+u_\xi=0, & \quad u(0,t)=0\label{stretch2} 
\end{align}
The problem (\ref{stretch1}) is now a variable coefficient problem. We may discretize it as 
\begin{align}
u_t+AP^{-1}Qu = -\frac{1}{2}[A_{11}P^{-1}]_0(u_0-0)\label{stretch_disc}
\end{align}
where $A_{ii}=\xi_x(x_i)$. To prove stability, we may freeze the coefficients and the proof is essentially the same as for $u_t+Du=\SAT$ above. However, when freezing the coefficients we neglect terms that have an effect, albeit bounded, on the energy estimate. Hence, this discretization may not be strictly stable.

Another option is to consider (\ref{stretch2}) and note that due to the non-singularity of $x(\xi)$, $A$ is positive definite. From this we may introduce a weighted $L^2$ norm, $\|v\|^2_{\xi}=v^TA^{-1}Pv$ for $P$ diagonal. In this norm, we obtain the growth $(\|v\|^2_{\xi})_t=-u_N^2$ for (\ref{stretch_disc}), which is the same as for (\ref{advec_stretch}) and hence the discretization is strictly stable. It was proven in (\cite{Svard04}) that the condition that $P$ is diagonal can not be circumvented.
 
This implies that while any SBP operator used to compute (\ref{stretch_disc}) is stable, it is only those with $P$ diagonal that can be proven to have the same time evolution as the continuous problem. For instance, if a block-norm SBP operator is used there may be small positive eigenvalues that cause a growth. This has been observed for the Euler equations on curvilinear meshes, where the block-norm operators require somewhat more artificial diffusion to converge to steady-state than their diagonal-norm counterparts, \cite{SvardMattsson05}. A similar phenomena was observed in \cite{KozdonDunham12(2)}.

\subsection{Dual consistency and superconvergence of functionals}\label{sec:dual}
In many cases, accurate solutions to the equations themselves might not be the primary target for a calculation. Typically, functionals computed from the solution, such as the lift and drag coefficients on a body in a fluid, potential or kinetic energy, or probabilities are of equal or even larger interest.  The importance of duality in the context of functionals was realized in adaptive mesh refinement, error analysis and optimal design problems, which has made the study of duality somewhat restricted to unstructured methods such as finite element (FEM), discontinuous Galerkin (DG) and finite volume (FVM) methods. 



Recently, however, it was shown in \cite{HickenZingg11,Hicken12,HickenZingg13} that the adjoint equations can be used for finite difference (FD) methods to raise the order of accuracy of linear functionals of numerical solutions generated by SBP-SAT schemes. In general, the truncation error of SBP-SAT discretization with diagonal norms is of order $2p$ in the interior and $p$ at the boundary, which results in a solution accuracy of $p + 1$. (See \cite{SvardNordstrom06}.) It was shown in \cite{HickenZingg11}, and later extended in \cite{BergNordstrom12},\cite{ZinggHick14}, that linear integral functionals from a diagonal norm dual-consistent SBP-SAT discretization retains the full accuracy of 2p. That is, the accuracy of the functional is higher than the solution itself. This superconvergent behavior was previously observed for FEM and DG methods but it had not been previously proven for finite difference schemes. 

It is important to note that dual consistency is a matter of choosing the coefficients in the SATs and it does not increase the computational complexity. Superconvergence of linear integral functionals hence comes for free. Free superconvergence is an attractive property of a dual consistent SBP-SAT discretization but the duality concept can also be used to construct new boundary conditions for the continuous problem. Research in this direction was initiated in \cite{BergNordstrom13}, and the procedure is under development for both the Euler and Navier-Stokes equations \cite{BNO14}.

\subsection{SBP-SAT in time}\label{sec:SBP-TIME}

In \cite{NordLun13} the SBP-SAT technique in space was extended to the time-domain. To present the main features of the methodology, we consider the simplest possible first order initial-value problem
\begin{equation}
u_t  =\lambda u,  \label{cont}
\end{equation}
with initial condition $u(0)=f$ and $0 \leq t \leq T$. Let $\lambda$ be a scalar complex constant representing an energy stable spatial
discretization of an IBVP. Hence, we assume that $Re(\lambda)<0$.
\begin{remark}
Note that the eigenvalues of difference approximations may be complex motivating our choice of $\lambda$. Furthermore, the size of the maximal eigenvalue of a difference approximation of a hyperbolic problem grows as $1/h$ and for a parabolic problem as $1/h^2$.
\end{remark}

The energy method (multiplying with the complex conjugated solution and integrating over
the domain) applied to (\ref{cont}) yields
\begin{equation}
|u(T)|^2-2 Re(\lambda)||u||^2=|f|^2,
\label{contest1}
\end{equation}
where $||u||^2=\int_0^T |u|^2dt$.  Since $Re(\lambda)<0$, the solution at the final time is bounded in terms of the initial data and we also obtain a bound on the (temporal) norm of the solution.

An SBP-SAT approximation of (\ref{cont}) reads
\begin{equation}
P^{-1}Q \vec{U}  =\lambda \vec{U} + P^{-1}(\sigma (U_0-f)) \vec{e_0}. \label{disc}
\end{equation}
The vector $\vec{U}$ contains the numerical approximation of $u$ at
all grid points in time. The matrices $P,Q$ form the differentiation
matrix with the standard SBP properties given in Definition \ref{def_sbp}
The penalty term on the
right-hand-side of (\ref{disc}) enforces the initial condition weakly
using the SAT technique at grid point zero using the unit
vector $\vec{e_0}=(1,0,...,0,0)^T$. 
\begin{remark}
The penalty term in (\ref{disc}) forces the discrete solution towards the initial data, i.e. $U_0 \neq f$ in general, but it is close. 
\end{remark}

The discrete energy method applied to (\ref{disc}) (multiplying from
the left with $\vec{U}^*P$, using the SBP properties and making the choice $\sigma=-1$ leads to
\begin{equation}
|\vec{U}_N|^2-2 Re(\lambda)||\vec{U}||^2_P=|f|^2-|U_0-f|^2.
\label{discest}
\end{equation}
The choice $\sigma=-1$ also makes the discretization dual consistent \cite{HickenZingg11},\cite{BergNordstrom12}, \cite{BergNordstrom13},\cite{ZinggHick14}. By comparing the continuous
estimate (\ref{contest1}) with (\ref{discest}) we see that the discrete
bound is slightly more strict than the continuous counterpart due to
the term $-|U_0-f|^2$ (which goes to zero with increasing accuracy).
\begin{remark}
Note that the estimate (\ref{discest}) is
  independent of the size of the time-step, i.e. the method is
  unconditionally stable.
\end{remark}
\begin{remark}
Sharp estimates like (\ref{discest}) can hardly
  be obtained using conventional local methods where only a few time
  levels are involved. One can argue, although no proof exist, that it
  can be done only with global methods.
\end{remark}

In \cite{NordLun13} it is shown that the new SBP-SAT method in time is high order accurate, unconditionally stable and together with energy stable semi-discrete approximations, it generates optimal fully discrete energy estimates.  In particular, for energy stable multi-dimensional system problems such as the Maxwells equations, the elastic wave equations and the linearized Euler and Navier-Stokes, fully discrete energy estimates for high order approximations can be obtained in an almost automatic way. 

In \cite{LunNord13}, it was shown how the SBP-SAT technique for time integration, originally formulated as a global method, can be used with flexibility as a one-step multi-stage method with a variable number of stages proportional to the order of the scheme, without loss of accuracy compared to the global formulation. This fact makes the SBP-SAT method highly competive, easier to program and significantly reduce the storage requirements. Classical stability results, including A- stability, L-stability and B-stability could also be proven using the energy method. In addition is was shown
that non-linear stability holds for diagonal norm operators when applied to energy stable initial value problems.

\section{SBP and other numerical schemes}\label{sec:other_schemes}

In this section, we will digress from the high-order finite difference scheme discussed so far. As stated in Definition \ref{def_sbp}, the SBP property is not tied to difference schemes but to the properties of certain matrices. A numerical derivative resulting from any numerical discretization technique can be expressed on matrix form and hence, it may or may not have the SBP property. If it has, we can employ the SAT framework to enforce boundary conditions and effectively reuse all the stability theory stemming from the research on high-order SBP finite difference schemes. 

\subsection{Finite volume schemes}
In \cite{NordstromForsberg03} it was shown that the unstructured node-centered finite volume first-derivative approximation is on SBP form. In \cite{NordstromBjorck01} also the structured cell-centered finite volume method is modified and extended to SBP form.  In \cite{SvardNordstrom04} it was shown that a common Laplacian approximation is also on SBP form and SBP consistent artificial diffusion was proposed in \cite{SvardGong06}. The entire framework for imposing boundary conditions with SAT terms was successfully used in \cite{ShoeybiSvard10} for finite volume schemes. The weak enforcement of boundary conditions created the possibility of hybrid flow solvers where finite volume and finite difference schemes are used in different domains. This greatly simplifies grid generation since complicated geometries can be wrapped in unstructured grids and in the bulk of the domain structured grids and the more accurate and efficient finite difference schemes are used. Hybrid schemes schemes have been developed in \cite{NordstromGong06, GongNordstrom07,NordstromHam09}.

\subsection{The discontinuous Galerkin method}
The SBP-SAT method can also be related the discontinuous Galerkin (dG) method. In fact, one may argue that it is the same methodology, only in a different technical setting. Consider the advection equation (\ref{const_advec}). We expand the solution in a polynomial $u = L^T(x)\,\, \vec{\alpha}(t)$ where $L(x)= (\phi_0, \phi_1,  \, \, ... \,\, \phi_N)^T$ and $ \vec{\alpha}(t) = (\alpha_0, \alpha_1, ... \, \alpha_N)^T$. By inserting $u$ into (\ref{const_advec}), multiplying each row with the basis functions $\phi_i$ and integrating over the domain we find
\begin{equation}\label{eqn: 2}
\int_0^1 L L^T dx\vec{\alpha}_t + a \int_0^1 L L_x^T dx \vec{\alpha} = 0, \quad \Rightarrow \quad P\vec{\alpha}_t + a Q \vec{\alpha} = 0.
\end{equation}
The matrices $P$ and $Q$ satisfy the SBP requirements in Definition \ref{def_sbp} if one chooses $\phi_i$ to be the $ith$ Lagrange polynomial. For more details, see \cite{Nordstrom06}.

How about imposing boundary conditions ?  To see the similarities we replace $Q$ in (\ref{eqn: 2}) by $-Q^T+B$ (integration-by-parts with SBP operators). We obtain the modified equation $P \vec{\alpha}_t+ a B \vec{\alpha} - a Q^T \vec{\alpha} = 0$. Next we use a common dG procedure and replace"what we have" ($\alpha_0$) with "what we want" ($g(t)$) and integrate back by replacing $Q^T$ with $-Q+B$. The final result is
\begin{equation}\label{eqn: 3}
P \vec{\alpha}_t+ a Q \vec{\alpha} = -a(\alpha_0-g(t))\vec e_0
\end{equation}
where $\vec e_0^T=(1,0...0)$. Clearly (\ref{eqn: 3}) includes a weak SAT term for the boundary condition and consequently the dG scheme presented above is on SBP-SAT form. For more details on relations to dG schemes, see \cite{GASS13}. Other striking similarities with SBP-SAT schemes can be found in the so called flux reconstruction schemes, where the penalty terms are constructed using specially designed polynomials. For a description of that method and similar ones, see \cite{JAME13}.

\section{Applications}\label{sec:app}

\subsection{The multi-dimensional Navier-Stokes equations}\label{sec:multiD_NS}

In this section, we will briefly sketch the procedure for the Navier-Stokes equations and we refer to the references for more details.  In general, the domain is not Cartesian. To be able to use finite difference schemes, it must be possible to subdivide the computational domain in several non-overlapping patches that cover the entire domain. In each subpatch, a curvilinear grid is constructed. The grid must be a sufficiently smooth to allow for a smooth transformation to a Cartesian mesh. That is, we assume that there are coordinates $0\leq \xi,\eta \leq 1$ and a transformation $x=x(\xi,\eta), y= y(\xi,\eta)$ to the physical space in each patch. Grid lines should be continuous across an interface but their derivatives need not. We remark that grid lines are allowed to end at an interface, i.e., the resolution of two neighboring blocks may be different, see \cite{MattssonCarpenter10} and also \cite{NissenKreiss12}. 

In each patch, the transformed system takes the form,
\begin{align}
Ju_t +(\tilde A_{1}u)_{\xi} +(\tilde A_{2}u)_{\eta}&=\epsilon(\tilde F^{v\xi}_{\xi}+\tilde F^{v\eta}_{\eta}),\quad 0< \xi,\eta < 1, \quad t>0\label{2D_NS_curv}\\
\tilde F^{v\xi}&=\tilde B_{11}u_{\xi}+\tilde B_{12}u_{\eta}\nonumber \\
\tilde F^{v\eta}&=\tilde B_{21}u_{\xi}+\tilde B_{22}u_{\eta}\nonumber 
\end{align}
augmented with suitable boundary, interface and initial conditions.

The process to obtain an energy estimate begins with a linearization followed by a symmetrization, see \cite{AbarbanelGottlieb81}.  To state the boundary conditions, the matrices, $\tilde A_{1,2}^{\pm}$, are needed. The diagonalization of these matrices were carried out in \cite{PulliamChaussee81}.
Derivation of energy estimates for the Euler and Navier-Stokes equations with various types of boundary and interface are found in \cite{NordstromCarpenter99,NordstromCarpenter01} where the specific difficulties associated with these equations are explained.  Furthermore, in \cite{SvardCarpenter07} far-field boundary conditions are derived and particular attention given to the difficult case of subsonic outflow. In \cite{SvardNordstrom08} solid wall and in \cite{BergNordstrom11} Robin wall boundary conditions are derived.  Also, recall that for curvilinear grids, SBP operators with diagonal $P$, have favorable stability properties as discussed in Section \ref{sec:coord} and \cite{Svard04}. The specific treatment of grid block interfaces is addressed in \cite{NordstromGong09}.

The diffusive terms in the Navier-Stokes equations may be discretized in a few different ways. In (\ref{2D_NS_curv}) above, they are treated as flux terms and hence they are approximated by repeated applications of the first-derivative approximation.  As discussed in section \ref{sec:sec_deriv} it may be advantageous from an accuracy point of view to use dedicated second-derivative approximations that use narrow stencils. To this end, we rewrite (\ref{2D_NS_curv}) as
\begin{align}
Ju_t +(\tilde A_{1}u)_{\xi} +(\tilde A_{2}u)_{\eta}&=\nonumber 
\epsilon(\tilde B_{11}u_{\xi\xi}+\tilde B_{12}u_{\eta\xi}+\tilde B_{21}u_{\xi\eta}+\tilde B_{22}u_{\eta\eta}).&\nonumber 
\end{align}
This derivation assumes that the matrices $\tilde B_{ij}$ are constant. We can then use the second-derivative operator given by $D_2=P^{-1}(-S^TMS+BS)$ on $u_{\xi\xi}$ and $u_{\eta\eta}$. (Recall that $BS$ is a matrix with the first and last rows non-zero and approximating the first-derivative.) The cross-derivative terms are still approximated with the first-derivative, $P^{-1}Q$, and we note that such an approximation is not "wide'' since they operate in different directions. To obtain an energy estimate, the $D_2$ and $D_1$ operators must relate in a certain way. Roughly speaking, the $D_2$ operator must be more diffusive than $D_1D_1$. Such $D_2$ operators are termed "compatible'' and can be found in \cite{MattssonSvard08}. In the case when the diffusion coefficient is not constant, $D_2$ must be a function of the diffusion coefficient. (C.f. $(\epsilon u_x)_x\approx (\epsilon_{i+1/2}(u_{i+1}-u_i)-\epsilon_{i-1/2}(u_i-u_{i-1}))/h^2$ in the interior of the scalar case with second-order approximation.) SBP versions of such approximations are found in \cite{Mattsson11}.
\begin{remark}
It should be mentioned here that the use of compact second derivatives complicates the programming since one have to keep track on whether one deals with a clean second derivative or a mixed types of derivative. Next one have to build the flux with these different components.  With the use of only first derivative operators, no such bookkeeping is necessary, one simply build the complete flux, and differentiate.
\end{remark}

Many times it is necessary to damp small oscillations. This is usually done by adding artificial dissipation to the scheme. For high-order schemes such diffusion terms are even order differences scaled such that they do not to reduce the order of accuracy of the scheme and maintain a grid independent damping on the highest modes. For instance, a fourth-order undivided second derivative is appropriate to add to a 4th-order accurate scheme. However, the artificial diffusion terms must be augmented with boundary closures that preserve the energy estimates. Without appropriate boundary closures, the artificial dissipation terms may destabilize the scheme. SBP conforming artificial diffusion operators can be found in \cite{MattssonSvard04}.


\subsubsection{Other time integration techniques and computational aspects}
\label{sec:compaspects}

To develop a code using the SBP-SAT schemes necessitates a number of choices apart from the spatial discretization scheme. For time-dependent problems, a time-discretization scheme must be chosen. In many of the computational results obtained in the references listed here a low-storage 3rd-order Runge-Kutta method has been used, \cite{KennedyCarpenter00}. However, the time-step constraint for an explicit method is often severe. If the stiffness is artificial, i.e., not reflecting time-scales in the actual solution, an implicit scheme may be more effective. We mention implicit-explicit Additive Runge-Kutta methods, \cite{KennedyCarpenter03},  as a viable option where the cost of implicit schemes may be kept to a minimum by using implicit schemes in  stiff regions only. An effort in this direction is found in \cite{ShoeybiSvard10}  in the context of SBP finite-volume schemes. (See also \cite{BijlCarpenter02,CarpenterKennedy05}.)

As a step towards effective implicit schemes, efficient steady-state solvers have been developed. While there are well-known and reliable libraries available for both linear and non-linear solvers there are still obstacles to overcome to use them in practice. One is to obtain the Jacobian matrix of the scheme. Already for second-order schemes it is cumbersome to derive and program exact Jacobians. For high-order methods even more so. Another option is approximate the Jacobian numerically. However, as the accuracy is increased by reducing the step size in the difference procedure, cancellation errors increase. To find the optimal value is not easy and the optimal value may not be accurate enough for fast convergence to steady state. In \cite{WeideSvard12}, dual numbers are used (see  \cite{FikeJongsma11}) to approximate the Jacobian numerically. This is an algebraic trick that rids the problem with cancellation and fast convergence is obtained.

Finally, we return to the spatial discretizations, which naturally has an effect on the convergence rate to steady state. It has already been mentioned that the use of diagonal-norm schemes ($P$ diagonal) increases the convergence speed and allows for less diffusion in the simulation, \cite{SvardMattsson05}. This has also been noted in \cite{KozdonDunham12(2), KozdonDunham12}, where it is further observed that splitting the convective terms has a favorable impact on convergence speed. Furthermore, the weak imposition of wall boundary conditions, i.e., the SAT procedure, increases the speed to steady state when compared to injecting the boundary conditions at the wall, \cite{NordstromEriksson12}.

\subsubsection{Computational Efficiency}

We have to a great extent discussed stability proofs of SBP-SAT schemes. This is the foundation to ensure robust and highly accurate simulations. Next, we will review the literature regarding the efficiency of SBP schemes. In \cite{SvardMattsson05,MattssonSvard07} steady and unsteady airfoil computations governed by the Euler equations are found. In particular it is demonstrated that low order schemes fail (on reasonable grids) to propagate flow structures over long distances. Airfoil computations (NACA 0012) with Navier-Stokes equations are found in \cite{SvardLundberg10} and a subtle case with $Ma=0.5$ and $Re=5000$ was successfully computed. The solution displays back flow but the wake remain stable. Inaccurate simulations would
produce an unstable wake. More simulations with the Navier-Stokes equations are found in \cite{SvardCarpenter07,SvardNordstrom08} for a flow around a cylinder. The focus in these articles is on the robustness and accuracy of far-field and wall boundary conditions. It is demonstrated that accurate shedding frequencies and separation points are obtained on course grids when the order of accuracy is sufficiently high. Two other evaluations of the SAT technique to enforce boundary conditions are found in \cite{Bodony09,Bodony10}. Here, the focus lies on aeroacoustic applications where minimal reflections from the boundaries are of the essence. As previously mentioned, the high-order SBP-SAT schemes was shown to be very competitive in terms of the use of computer resources, compared with other high-order techniques, see \cite{WeideSvard12}.

In \cite{ShoeybiSvard10}, an unstructured  finite volume SBP scheme was implemented for the Euler and Navier-Stokes equations. Implicit-explicit time integration were employed and the implicit part adaptively associated with stiff regions in the problem. This procedure greatly reduced the computational cost and the efficiency was demonstrated on a number of test cases including LES around a cylinder. 

\subsection{Other applications}
So far, we have mainly discussed SBP-SAT schemes applied to the compressible Navier-Stokes equations. Here, we will briefly summarize other areas where SBP-SAT schemes have been successfully used. In \cite{SvardMishra11}, SBP operators were used to solve the compressible Euler equations with a stiff source term modeling a reaction equation.  Special boundary conditions that automatically produced divergence free solutions were derived and implemented in stable SBP schemes for the incompressible Navier-Stokes equations in \cite{NordstromMattsson07}. The incompressible Navier-Stokes equations were also considered in \cite{HamMattsson07} where an accuracy evaluation of SBP-SAT finite volume schemes was the main objective.

Some applications in aeroacoustics, using the linearized Euler equations, are found in \cite{Muller08,MullerYee02} and with focus on human phonation in \cite{LarssonMuller09}.  Fluid-structure interaction problems were studied in \cite{NordstromEriksson10}, and it was shown that 
the weak coupling in SBP-SAT schemes was the coupling procedure that lead to the best result for this very sensitive problem. In \cite{LindstromNordstrom10} conjugate heat transfer between a solid and a fluid was analyzed using SBP schemes. Moreover, in \cite{NordstromBerg13} the conjugate heat transfer was analyzed further by using the Navier-Stokes equations only (no heat equation was involved) but in combination with a special interface conditions. The resulting well-posed problem was implemented using an SBP-SAT scheme. 

Another area particularly suited for  high-order SBP schemes is wave propagation. In \cite{MattssonNordstrom06,MattssonHam08,MattssonHam09} the wave equation on second-order form was analyzed; boundary conditions as well as interface conditions between media with different propagation speeds were derived and SBP-SAT schemes for solving the problem proposed.

In numerical relativity, the Einstein equations constitute a set of equations similar to the second-order wave equation and much effort has gone into designing effective SBP-SAT schemes. A sample of articles in this field is:  \cite{SchnetterDiener06,DienerDorband07,DorbandBerti06,ZinkSchnetter08,MattssonParisi10}.
Furthermore, the SBP techniques have been applied to Maxwell's equations of electromagnetism (\cite{NordstromGustafsson03}) and the magnetic induction equations (\cite{KoleyMishra09,KoleyMishra12,MishraSvard10}). Recently, they have also been used in wave propagation problems in geophysics \cite{KozdonDunham12(2), KozdonDunham12} and medicine \cite{AmNord13}.

The SBP-SAT methodology has also been used in the emerging area of Uncertainty Quantification (UQ), which combines computational mathematics and mathematical statistics. Unlike the conventional approach using Monte Carlo simulations that require a large number of runs to obtain reliable statistics (and hence calls for extremely fast solvers), UQ is designed to solve directly for the solution and its associated statistics. Hence, UQ is potentially very effective and SBP-SAT schemes produce accurate numerical solutions to the resulting PDE system.  Problems in UQ solved the SBP-SAT technique are found in \cite{PettIaccNord09,PettNordIacc10,PettDooNord13,PettIaccNord13,PIN14}.


\section{Non-linear stability}\label{sec:non_lin}

The SBP-SAT schemes are often used to approximate non-linear PDEs, such as the Euler and Navier-Stokes equations. Before discussing non-linear stability, we will make a brief justification for the linear stability proofs obtained with SBP-SAT schemes. The Navier-Stokes equations can be linearized around a smooth solution (assuming it exists) which gives a variable coefficient linear problem in the perturbed variables. If the solution of this problem is proven to be bounded then the non-linear problem is stable with respect to small perturbations. A similar argument is then made for a numerical discretization of the non-linear equations, \cite{Strang64}. The simplification process is often taken one step further and the "constant coefficient'' problem is analyzed. This means that the variable coefficients in the PDE  are "frozen'' prior to the analysis. This approach is justifiable, at least if the energy method is used to prove stability, see \cite{KreissLorenz, MishraSvard10}. (If other techniques are used to prove stability the situation is much more subtle, see \cite{Michelson83,Michelson87,Wade90,StrikwerdaWade88}). This line of argument justifies the convergence observed in subsonic regimes where non-linear effects are negligible. Mild non-linear effects can be kept at bay by the use of high-order diffusion terms, \cite{MattssonSvard04,DienerDorband07}.


Strong non-linear effects in fluid mechanics are often associated with shocks and the generation of non-smooth solutions for which the above linear arguments do not hold. However, it is possible to derive $L^2$ bounds without prior linearization. Unfortunately, an $L^2$ bound is not enough to imply convergence for a non-linear PDE. Nevertheless, stability of some sort is required and many different paths have been taken. 

A popular choice of finite difference schemes are WENO type schemes, originally derived in \cite{JiangShu96}, that produce approximate solutions with sharp shock resolution and formally high order of accuracy. However, even linear stability without boundary conditions for such schemes are a subtle matter, \cite{WangSpiteri07,MotamedMacdonald11}. In two recent papers, the concepts of WENO and SBP were merged yielding linearly stable schemes for initial-boundary-value problems,   \cite{YamaleevCarpenter09,FisherCarpenter11}.

Another approach was taken in \cite{GerritsenOlsson96,GerritsenOlsson98} were they used a splitting of the flux function based on entropy arguments (see \cite{OlssonOliger94}). Under certain assumptions, including a bounded wave speed, non-linear energy estimates were obtained. The drawback of this approach is that it usually introduces a non-conservative term which may cause an incorrect shock location. However, for scalar problems it was shown in \cite{FisherCarpenter12} that most splitted schemes are in fact conservative, and hence the non-conservative terms will not cause a drift in the shock location. 


A successful way to prove non-linear stability for conservation laws is reviewed in \cite{Tadmor03}. The idea is to build a scheme that prevents the entropy from growing unboundedly and consequently they are termed \emph{entropy stable}. The theory was developed for Cauchy problems and a global entropy estimate obtained using summation-by-parts (without boundaries). In \cite{SvardMishra12}, second-order SBP schemes were put in the entropy stability framework and non-linear stability proven for initial-boundary-value conservation laws, including the Euler equations. The theory was extended to 3rd-order accurate SBP schemes in \cite{Svard12}. Non-linearly entropy stable boundary conditions were revisited in \cite{SvardOzcan13} and here stability proofs for far-field, solid wall and interfaces were given for the Euler equations. In the WENO framework, entropy stable schemes have been derived in \cite{FisherCarpenter13} with SBP boundary closures such that the boundary conditions in \cite{SvardOzcan13,SvardMishra12} are readily applicable.

\section{Summary and Outlook}

SBP-SAT schemes have reached a mature state where it is relatively straightforward to apply the method to new problems, as demonstrated by the vastly different applications addressed to date. The mathematical foundation of SBP-SAT schemes, with the most important property being that convergence can be proved for linear or linearized problems, gives credibility to the numerical simulations that ad hoc methods can not offer. The possibility to evaluate the size of numerical errors should be of paramount importance in both engineering and scientific simulations.

The key to the stability proofs is of course the boundary treatments of the SBP-SAT schemes. This is the feature that also makes the methodology versatile. It allows for stable and accurate couplings with other type of methods (hybrid schemes) and it is also the key to connect different models in a stable manner. (E.g. fluid-structure interaction.)

The remarks above outlines three paths for future research:
\begin{itemize}
\item Applying SBP-SAT to new problems which to a large extent involves analysis of boundary conditions.
\item Connecting different models which calls for analysis of interface conditions and the subsequent analysis to derive stable SAT connections
\item Derivation of new hybrid schemes like FD-DG which would increase the accuracy in the unstructured part.
\item Combining the SBP-SAT technique in space and time to obtain stable fully discrete approximations of IBVP.
\end{itemize}
We also mention the extension of SBP-SAT schemes to non-linear conservation laws (in the non-linear regime).  The development of non-linear stability and convergence theory is of course  hampered by the lack of mathematical knowledge. Existence results which are crucial when proving convergence are scarce for non-linear PDEs. However, the SBP-SAT schemes have many important properties that make them promising candidates in the developlment of non-linear stability theory and some steps have been taken in this direction.

\bibliographystyle{alpha}


\newcommand{\etalchar}[1]{$^{#1}$}

\end{document}